\documentclass[a4,12pt]{amsart}
\usepackage{amssymb,amstext,amsmath,amscd,amsthm,amsfonts,enumerate,latexsym, comment}
\usepackage{color}
\usepackage[dvipdfmx]{graphicx}

\usepackage[all]{xy}

\theoremstyle{plain}
\newtheorem{thm}{Theorem}[section]

\newtheorem*{thm*}{Theorem}
\newtheorem*{cor*}{Corollary}

\newtheorem{prop}[thm]{Proposition}

\newtheorem{lem}[thm]{Lemma}
\newtheorem{cor}[thm]{Corollary}

\newtheorem*{claim*}{Claim}

\theoremstyle{definition}
\newtheorem{defn}[thm]{Definition}

\newtheorem{ex}[thm]{Example}

\newtheorem{prob}[thm]{Problem}
\newtheorem{rem}[thm]{Remark}
\newtheorem*{ac}{Acknowledgments}
\theoremstyle{remark}

\numberwithin{equation}{thm}

\def\Min{\operatorname{Min}}

\def\Max{\operatorname{Max}}
\def\Spec{\operatorname{Spec}}

\def\m{\mathfrak m}

\newcommand{\rme}{\mathrm{e}}

\newcommand{\rmH}{\mathrm{H}}

\newcommand{\rmQ}{\mathrm{Q}}

\newcommand{\calF}{\mathcal{F}}

\newcommand{\calS}{\mathcal{S}}

\newcommand{\calX}{\mathcal{X}}

\newcommand{\fka}{\mathfrak{a}}
\newcommand{\fkb}{\mathfrak{b}}

\newcommand{\fkp}{\mathfrak{p}}

\newcommand{\mapright}[1]{%
\smash{\mathop{%
\hbox to 1cm{\rightarrowfill}}\limits^{#1}}}

\newcommand{\mapleft}[1]{%
\smash{\mathop{%
\hbox to 1cm{\leftarrowfill}}\limits_{#1}}}

\def\depth{\operatorname{depth}}

\def\Ass{\operatorname{Ass}}

\def\Spec{\operatorname{Spec}}

\title[Decomposition of integrally closed ideals in Arf rings]{Decomposition of integrally closed ideals in Arf rings}

\author[R. Isobe]{Ryotaro Isobe}

\address{Department of Mathematics and Informatics, Graduate School of Science, Chiba University, Yayoi-cho 1-33, Inage-ku, Chiba, 263-8522, Japan}
\email{r.isobe.math@gmail.com}


\thanks{The author was partially supported by JSPS KAKENHI Grant Number JP20J10517 and 21K13767.}

\begin{document}

\maketitle

\setlength{\baselineskip} {15.2pt}

\begin{abstract}
This study investigates the structure of Arf rings.
From the perspective of ring extensions, a decomposition of integrally closed ideals is given. Using this, we present a kind of  their prime ideal decomposition in Arf rings, and determine their structure in the case where both $R$ and the integral closure $\overline{R}$ are local rings.

\end{abstract}



\section{Introduction}\label{intro}
The purpose of this paper is to give a decomposition of integrally closed ideals that extends the property of discrete valuation rings to Arf rings. 

In 1971, J. Lipman introduced the notion of Arf rings for Cohen-Macaulay semi-local rings $R$ satisfying that every localization $R_M$ at a maximal ideal $M$ is dimension one.
Arf rings were originally studied in the classification of certain singular points of plane curves by C. Arf \cite{Arf}, and Lipman generalized them by extracting the essence of the rings.
A typical example of an Arf ring is a Cohen-Macaulay local ring of dimension one with multiplicity at most two {\cite[Example, page 664]{L}}. It is also known that a semi-normal Cohen-Macaulay local ring of dimension one is Arf (see {\cite[Theorem 4.4]{C}}, {\cite[Lemma 4.1]{L}}, \cite{GT}).
Using this notion, it was proved that if $R$ is a one-dimensional complete Noetherian local domain with an algebraically closed residue field of characteristic zero and if $R$ is {\it saturated} in the sense of O. Zariski \cite{Z}, then $R$ has minimal multiplicity (i.e., the embedding dimension of $R$ equals the multiplicity of $R$). This result depends on the fact that such a ring is an Arf ring.
In \cite{L}, Arf rings are characterized in terms of the stability of integrally closed open ideals and the behavior of the blow-ups (\cite[Theorem 2.2]{L}), which in turn tells us that every Arf local ring has minimal multiplicity.

Additionally, the Arf property is strongly related to the strict closedness of rings, which was given in the same paper \cite{L}.  
For the sake of simplicity, let $R$ be a Cohen-Macaulay local ring of dimension one.  
As is mentioned in \cite{L}, it was conjectured that $R$ is an Arf ring if and only if $R=R^*$, where $R^*$ is an intermediate ring between $R$ and the integral closure $\overline{R}$, {\it the strict closure of $R$}, consisting of those elements $\alpha$ in $\overline{R}$ such that $\alpha \otimes 1=1\otimes \alpha$ in $\overline{R}\otimes_R \overline{R}$.
This conjecture was solved in \cite{L} when $R$ contains a field ({\cite[Proposition 4.5, Theorem 4.6]{L}}), and was settled in \cite{C} for its general case ({\cite[Theorem 4.4]{C}}).
The core of the theory is that the ring structure is explored in connection with the behavior of ring extensions. The reader may consult with \cite{L, RGGB, AS, CCGT, C} about further study of Arf rings and strictly closed rings.

Recently, E. Celikbas, O. Celikbas, C. Ciuperc\u{a}, N. Endo, S. Goto, this author, and N. Matsuoka \cite{C} introduced the notion of weakly Arf rings by weakening the defining conditions of  Arf rings, and extended the theory over arbitrary commutative rings. 
It is known by \cite[Corollary 4.6]{C} that if $R$ is a Noetherian ring that satisfies the Serre's $(S_2)$ condition and $ {\rm ht}_R M\ge 2$ for every $M\in\Max R$, then $R$ is a weakly Arf ring if and only if $R_{\fkp}$  is an Arf ring for every $\fkp \in \Spec R$ with ${\rm ht}_R \fkp=1$, which is also equivalent to $R=R^*$.
Thus,  in this assumption, because $R=\overline{R}$ if and only if $R_{\fkp}$  is a discrete valuation ring (DVR) for every $\fkp \in \Spec R$ with ${\rm ht}_R \fkp=1$ (i.e. $R$ is normal), we can develop the theory of weakly Arf rings parallel to that of integrally closed rings by replacing a DVR with an Arf ring, and the integral closure $\overline{R}$ with the strict closure $R^*$. 
Hence, in this study, to build the foundation for this theory,  we explore the structure of Arf rings compared with the structure of DVRs. 
The main result of this paper is stated as follows:

\begin{thm}[Corollary \ref{5.6}]
Let $R$ be an Arf ring, and let $I$ be an integrally closed ideal in $R$ that contains a non-zerodivisor on $R$.
Then, we can construct a certain tower $R=R_{(I, 0)}\subseteq R_{(I, 1)}\subseteq \cdots \subseteq R_{(I, n)}$ of Arf rings in $\overline{R}$ and an integrally closed ideal $I_i$ in $R_{(I, i)}$ for each $0\le i \le n$, such that $I=\sqrt{I_0}\sqrt{I_1}\cdots\sqrt{I_n}$, where $I_0=I$.
\end{thm}

If $I\neq R$, then we can choose $n\ge 0$ such that $I=\sqrt{I_0}\sqrt{I_1}\cdots\sqrt{I_n}$ and $\sqrt{I_{i}}=\prod_{\fkp\in V(I_i)}\fkp$ for each $0\le i\le n$. Therefore, corresponding to the fact that every non-zero ideal in a DVR is the power of the maximal ideal, $I$ can be decompose into the products of maximal ideals in  rings $R$, $R_{(I, 1)}, R_{(I, 2)}, \cdots, R_{(I, n)}$.

We now explain the organization of this paper. In Section 2, based on \cite{L}, we introduce the definition and basic properties of Arf rings, which we subsequently need. Section 3 provides the key property in this study regarding elements in the set $\calX_R$ of integrally closed ideals in $R$ that contain a non-zerodivisor on $R$.  
In Section 4, we prove the main result of this paper.
In Sections 5, we study Corollary \ref{5.6} in detail. This section studies the case where $R$ and $\overline{R}$ are local rings, while introducing some examples.

Throughout this paper, unless otherwise specified, let $R$ be an arbitrary commutative ring, let $W(R)$ be the set of non-zerodivisors on $R$, and let $\calF_R$ be the set of {\it open} ideals in $R$, that is, the ideals of $R$ that contain a non-zerodivisor on $R$. For an ideal $I$ in $R$, $\overline{I}$ denotes the integral closure of $I$ in $R$, and set $\calX_R=\{ I\in \calF_R \mid \overline{I}=I \}$. We set $X:Y=\{a\in\rmQ(R) \mid aY\subseteq X\}$ for $R$-submodules $X$ and $Y$ of the total ring of fractions $\rmQ(R)$.

When $R$ is a Noetherian local ring with the maximal ideal $\m$, let $v(R)$ (resp. $\rme(R)$) denote the embedding dimension of $R$ (resp. the multiplicity of $R$ with respect to $\m$).


\section{Arf rings}\label{sec2}

In this section, we introduce the definition and basic properties of Arf rings, based on \cite{L}. For an arbitrary commutative ring $R$, let $W(R)$ be the set of non-zerodivisors on $R$. We denote by $\calF_R$ the set of ideals in $R$ that contain a non-zerodivisor on $R$.

First, we define the algebra $R^I$ for each $I\in\calF_R$. Let $I\in \calF_R$. We consider the tower of $R$-algebras as follows:
$$R\subseteq I:I\subseteq I^2:I^2\subseteq \cdots \subseteq I^n:I^n \subseteq \cdots \subseteq \rmQ(R).$$
We set $R^I=\bigcup_{n\ge0}[I^n:I^n]$. Then, $R^I$ is an intermediate ring between $R$ and $\rmQ(R)$. If $a\in I$ is a reduction of $I$, that is, $I^{r+1}=aI^r$ for some $r\ge 0$, then $R^I=I^n:I^n$ for any $n\ge r$ and we have 
$$R^I=R\left[ \tfrac{I}{a}  \right] = \frac{I^r}{a^r},\ \text{where} \  \frac{I}{a}= \{ \frac{x}{a} \mid x\in I \}\subseteq \rmQ(R).$$
In particular, the following holds true.

\begin{lem}\label{2.1}
Let $I\in\calF_R$. Suppose that there exists $a\in I$ and $r\ge0$ such that $I^{r+1}=aI^r$. Then, the following conditions are equivalent:
\begin{enumerate}[$(1)$]
\item
$R^I=I:I.$
\item
$I:I=\frac{I}{a}.$
\item
$I^2=aI.$
\end{enumerate}
Therefore, condition $(3)$ is independent of the choice of reductions $a\in I$. An ideal $I\in \calF_R$ that satisfies condition $(3)$ is called a stable ideal.
\end{lem}

\begin{proof}
The implication $(3) \Rightarrow (1)$ is true. Let $A=I:I$. The inclusion $A\subseteq \frac{I}{a}$ always holds. 

$(1)\Rightarrow(2)$ Because $R^I=A\subseteq \frac{I}{a}\subseteq R\left[ \frac{I}{a} \right]=R^I$, we have $A=\frac{I}{a}$.

$(2)\Rightarrow(3)$ Because $I=IA=\frac{I^2}{a}$, we have $I^2=aI$.
\end{proof}

In what follows, unless otherwise specified, we assume that $R$ satisfies the condition $(A)$.

\begin{enumerate}
\item[$(A)$]
$R$ is a Cohen-Macaulay semi-local ring of dimension one with $\dim R_M=1$ for every $M\in\Max R$.
\end{enumerate}

Under this assumption, the notion of Arf rings is defined as follows. We set $\calX_R=\{I\in\calF_R \mid \overline{I}=I \}$.

\begin{defn}[\cite{L}]\label{2.2}
$R$ is called an {\it Arf ring} if the following conditions hold:
\begin{enumerate}[$(1)$]
\item
For every $I\in\calX_R$, there exists an element $a\in I$ such that $I=\overline{aR}$, that is, $I^{n+1}=aI^n$ for some $n\ge0$.  
\item
If $x, y, z \in R$ such that $x\in W(R)$ and $\frac{y}{x}, \frac{z}{x}\in \overline{R}$, then $\frac{yz}{x}\in R$.
\end{enumerate}
\end{defn}

Arf rings are characterized in terms of the stability of the integrally closed ideals in $\calX_R$.

\begin{thm} [{\cite[Theorem 2.2]{L}}]\label{2.3}
The following conditions are equivalent: 
\begin{enumerate}[$(1)$]
\item
$R$ is an Arf ring.
\item
For every $I\in \calX_R$, there exists $a\in I$ such that $I^2=aI$.

\end{enumerate}
\end{thm}

For a Noetherian local ring $R$,  $v(R)$ denotes the embedding dimension of $R$, and $\rme(R)$ denotes the multiplicity of $R$ with respect to the maximal ideal. When $R$ is an Arf ring, the equation $v(R_M)=\rme(R_M)$ holds for every $M\in\Max R$ (i.e., $R_M$ has minimal multiplicity), because $M^2=aM$ for some $a\in M$. In contrast, even if $R_M$ has minimal multiplicity for every $M\in \Max R$, $R$ is not necessarily an Arf ring.

The structure of $R^I$ plays an important role in understanding the structure and properties of an Arf ring $R$.  Let us begin with the following:

\begin{lem}[c.f. {\cite[Lemma 2.3]{L}}]\label{2.4}
Let $R$ be an arbitrary commutative ring and let $I\in \calX_R$. Suppose that there exists $a\in I$ such that $I=\overline{aR}$ and $I^2=aI$. We set $A=I:I$. Then, we have 
$$\calX_A=\{ \frac{\fka}{a} \mid \fka\in\calX_R,\ \fka\subseteq I\}.$$ 
\end{lem}
\begin{proof}
Suppose that $\fka\in \calX_R$ and $\fka\subseteq I$. We set $\fkb=\frac{\fka}{a}$. First, we show that $\fkb\in\calF_A$. Because 
 $I=\overline{aR}=a\overline{R}\cap R$, we have $A=\frac{I}{a}\subseteq \overline{R}$. Additionally, we have
$$\fka A=\frac{\fka I}{a}\subseteq \frac{I^2}{a}=I\subseteq R.$$
Then, we get
$$\fka A\subseteq \fka\overline{R}\cap R\subseteq \overline{\fka\overline{R}}\cap R=\overline{\fka}=\fka,$$
which implies that $\fka A=\fka$. Therefore, we also have $\fkb A=\fkb$, so that $\fkb$ is an ideal of $A$. $\fka\subseteq \fkb$ implies that $\fkb\in\calF_A$.  
Second, we show that $\fkb\in\calX_A$. Take $\alpha\in\overline{\fkb}\subseteq A$. Then
$$\alpha^n+b_1\alpha^{n-1}+\cdots+b_{n-1}\alpha+b_n=0$$
for some $n\ge 1$ and $b_i\in \fkb^i$ $(1\le i\le n)$. Because $a\alpha\in a\fkb=\fka\subseteq R$, we obtain
$$(a\alpha)^n+ab_1(a\alpha)^{n-1}+\cdots+a^{n-1}b_{n-1}(a\alpha)+a^nb_n=0\ \ \text{and}$$
$$a^ib_i\in a^i\fkb^i=\fka^i\ \ (1\le i\le n),$$
which implies that $a\alpha\in\overline{\fka}=\fka$. Therefore, $\alpha \in \frac{\fka}{a}=\fkb$, which yields $\overline{\fkb}=\fkb$.

Conversely, let $\fkb\in\calX_A$ and set $\fka=a\fkb$. Then, $\fka\in\calF_R$ and $\fka \subseteq I$. For $x\in \overline{\fka}$, we take the equation
$$x^n+c_1x^{n-1}+\cdots+c_{n-1}x+c_n=0$$
where $n\ge 1$ and $c_i\in\fka^{i}$ $(1\le i\le n)$. Then, because $x\in \overline{\fka}\subseteq \overline{I}=I$, we get $\frac{x}{a}\in \frac{I}{a}=A$ and
$$\left(\frac{x}{a}\right)^n+\left(\frac{c_1}{a}\right)\left(\frac{x}{a}\right)^{n-1}+\cdots+\left(\frac{c_{n-1}}{a^{n-1}}\right)\left(\frac{x}{a}\right)+\left(\frac{c_n}{a^n}\right)=0,$$
where $\frac{c_i}{a^i}\in \fkb^i$ $(1\le i \le n)$, which implies that $\frac{x}{a}\in\overline{\fkb}=\fkb$. Therefore, we have $x\in a\fkb=\fka$, so that $\fka\in \calX_R$, as desired. 
\end{proof}

When $R$ satisfies condition $(A)$, $R^I$ also satisfies the condition for every $I\in \calX_R$. Therefore, by using Lemma \ref{2.1}, Theorem \ref{2.3}, and Lemma \ref{2.4}, we obtain the following:
  
\begin{cor}\label{2.5}
If $R$ is an Arf ring, $R^I(=I:I)$ is also an Arf ring for every $I\in \calX_R$.
\end{cor}

\begin{cor}\label{2.6}
If $R$ is an Arf ring, $v([R^I]_M)=\rme([R^I]_M)$ holds for every $I\in \calX_R$ and $M\in \Max R^I$.
\end{cor}

\begin{rem}
$R^I$ is not necessarily a local ring, even if $R$ is a local ring.
\end{rem}

We denote by $\operatorname{J}(R)$ as the Jacobson radical of $R$. Then, $\operatorname{J}(R)\in\calX_R$. Here, we set $R_1=R^{\operatorname{J}(R)}$ and define recursively 
\begin{center}
$R_n=$
$
\begin{cases}
\  R & \ \ (n=0)\\
\ [R_{n-1}]_1 & \ \ (n>0)\\
\end{cases}
$
\end{center}
for each $n\ge 0$. Then, we get the tower
$$R=R_0\subseteq R_1\subseteq \cdots \subseteq R_n \subseteq \cdots \subseteq \overline{R}$$
of rings, and every $R_n$ satisfies condition $(A)$. Moreover, every $R_n$ is an Arf ring if so is $R$. 

By using this tower of rings, we get another characterization of Arf rings.

\begin{thm}[{\cite[Theorem 2.2]{L}}]\label{2.7}
The following conditions are equivalent:
\begin{enumerate}[$(1)$]
\item
$R$ is an Arf ring.
\item
$v([R_n]_M)=\rme([R_n]_M)$ holds for every $n\ge0$ and $M\in\Max R_n$.
\end{enumerate}
\end{thm} 

\begin{ex}[{\cite[Example, P. 664]{L}}]\label{2.8}
Let $R$ be a Cohen-Macaulay local ring with $\rme(R)\le 2$. Then, $v([R_n]_M)=\rme([R_n]_M)\le 2$ holds for every $n\ge0$ and $M\in\Max R_n$. Therefore, $R$ is an Arf ring according to Theorem \ref{2.7}.
\end{ex}


\section{Integrally closed ideals}\label{sec3}

The purpose of this section is to show the key properties of ideals in $\calX_R$, which we need in this study. Throughout this section, let $R$ be a Noetherian ring. 
We begin with the following:

\begin{lem}\label{3.1}
$I:I=(R:I)\cap \overline{R}$ for every $I\in\calX_R$.
\end{lem}
\begin{proof}
It is sufficient to show that $I:I\supseteq (R:I)\cap \overline{R}$. Let $\varphi\in \overline{R}$ and $\varphi\cdot I\subseteq R$. Then, we have
$$a\varphi\in a\overline{R}\cap R=\overline{aR}\subseteq \overline{I}=I$$
for any $a\in I\cap W(R)$. Because $I$ is generated by non-zerodivisors on $R$, we obtain $\varphi\cdot I\subseteq I$, which yields $\varphi\in I:I$.
\end{proof}

\begin{prop}\label{3.2}
Let $I, J\in\calX_R$. If $I\subseteq J$, then $I:I\supseteq J:J$, and $I$ is an ideal of the ring $J:J$. 
\end{prop}

\begin{proof}
The inclusion $I\subseteq J$ implies $R:I\supseteq R:J$, so that
$$I:I=(R:I)\cap \overline{R}\supseteq (R:J)\cap \overline{R}=J:J$$
from Lemma \ref{3.1}. Because $I$ is an ideal of $I:I$ and $I:I\supseteq J:J$, $I$ is also an ideal of $J:J$.
\end{proof}

The following theorem is the key to our decomposition:

\begin{thm}\label{3.3}
Let $I, J\in\calX_R$ and suppose that $I\subseteq J$. We set $B=J:J$. Then, the following assertions hold true:

\begin{enumerate}[$(1)$]
\item
$I:J\in \calX_B$.
\item
Suppose that there exists $b\in J$ such that $J^2=bJ$. Then, $I:J=\frac{I}{b}$. Therefore, we have $I=J\cdot(I:J)$ and $I:J\cong I$ as $B$-modules. 
\end{enumerate}
\end{thm}

\begin{proof}
$(1)$   Note that $I:J\subseteq J:J=B$ and 
$$J\cdot[B\cdot (I:J)]=JB\cdot (I:J)=J\cdot (I:J)\subseteq I.$$
Thus, $I:J$ is an ideal of $B$. Because $I\subseteq I:J$, we have $I:J\in\calF_B$.

 Next, we show that $\overline{I:J}=I:J$ in $B$. 
 This follows from the proof of \cite[Remark 1.3.2 (2)]{SH}. In fact,
 let $\alpha\in\overline{I:J}\subseteq B$. We take the equation
 $$\alpha^n+a_1\alpha^{n-1}+\cdots + a_{n-1}\alpha+a_n=0,$$
 where $n\ge 1$ and $a_i\in (I:J)^i$ $(1\le i \le n)$. For any $x\in J$, we have $x\alpha \in JB=J \subseteq R$, 
 $$(x\alpha)^n+xa_1(x\alpha)^{n-1}+\cdots + x^{n-1}a_{n-1}(x\alpha)+x^na_n=0\ \ \text{and}$$
 $$x^ia_i\in x^i\cdot(I:J)^i\subseteq I^i\ \ (1\le i\le n),$$
so that $x\alpha \in\overline{I}=I$. This implies that $\alpha\cdot J\subseteq I$, which yields $\alpha \in I:J$ and $\overline{I:J}=I:J$, as desired.

$(2)$
Suppose that $J^2=bJ$ for some $b\in J$. Then, $J=bB$, and the inclusion $I:J\subseteq \frac{I}{b}$ holds. Additionally, because $I$ is an ideal of $B$ by Proposition \ref{3.2}, we have
$$J\cdot \frac{I}{b}=bB\cdot \frac{I}{b}=IB=I,$$
which yields $\frac{I}{b}=I:J$ and $I=J\cdot (I:J)$. We also have $I:J=\frac{I}{b}\cong I$ as $B$-modules.
\end{proof}
 
 For any $I\in\calF_R$, we have $\sqrt{I}\in \calX_R$, and the following holds true:

\begin{prop}\label{3.4}
Let $I\in \calX_R$. Suppose that $I\neq R$, and there exist $a, b\in W(R)$ such that $I=\overline{aR}$ and $\sqrt{I}=\overline{bR}$. We set ${\rmH}_1(R)=\{ \fkp\in\Spec R \mid {\rm ht}_R \fkp =1  \}$. If $R$ satisfies $(S_2)$ and $\Ass_R R/I\subseteq {\rmH}_1(R)$, then we have $\sqrt{I}=\underset{\fkp\in V(I)\cap {\rmH}_1(R)}{\bigcap}\fkp$, and the following conditions are equivalent:  

\begin{enumerate}[$(1)$]
\item
$I=aR$.
\item
$\sqrt{I}=bR$.
\item
$R_\fkp$ is a DVR for any $\fkp \in \Ass_R R/I$.
\end{enumerate}

\end{prop}
\begin{proof}
Because $\Ass_R R/I \subseteq {\rmH}_1(R)$, we have $\Ass_R R/I =V(I)\cap {\rmH}_1(R)$. Thus, $\sqrt{I}=\underset{\fkp\in V(I)\cap {\rmH}_1(R)}{\bigcap}\fkp$. By using \cite[Theorem 1.1]{G}, it is sufficient to show that $\Ass_R R/aR=\Ass_R R/I=\Ass_R R/bR$. 
In fact, let $\fkp\in \Ass_R R/aR$. Because $\depth R_{\fkp}=1$ and $R$ satisfies $(S_2)$, we have $\fkp\in{\rmH}_1(R)$ and
$$\Ass_R R/aR=V(aR)\cap {\rmH}_1(R)=V(I)\cap {\rmH}_1(R)=\Ass_R R/I.$$ 
Similarly, we have 
$$\Ass_R R/bR=\Ass_R R/\sqrt{I}=V(\sqrt{I})\cap {\rmH}_1(R)=V(I)\cap {\rmH}_1(R),$$
which yields $\Ass_R R/aR=\Ass_R R/I=\Ass_R R/bR$, as desired.
\end{proof}

We will use this proposition in Section 4.


\section{Main result: Equation $I=\sqrt{I_0}\sqrt{I_1}\cdots\sqrt{I_n}$}\label{sec5}

In the following sections, 
we assume that $R$ is an Arf ring. We fix $I\in\calX_R$ and set $B=I:I$. Then, if $I\neq R$, we can write 
$$\sqrt{I}=\underset{\fkp\in V(I)}{\bigcap}\fkp=\underset{\fkp\in V(I)}{\prod}\fkp.$$ 

\begin{defn}\label{4.1}
We set
$$R_{(I, 1)}=\sqrt{I}:\sqrt{I} ~~\text{and} ~~ I_1=I:\sqrt{I}.$$

By Corollary \ref{2.5} and Theorem \ref{3.3}, $R_{(I, 1)}$ is also an Arf ring, $I_1\in \calX_{R_{(I, 1)}}$, and $I=\sqrt{I}\cdot I_1$. 
For each $n\ge 0$, we define recursively


$$
R_{(I, n)} = 
\begin{cases}
R & (n=0)\\
\sqrt{I_{n-1}}:\sqrt{I_{n-1}}=[R_{(I, n-1)}]_{(I_{n-1}, 1)} & (n>0),
\end{cases}
$$
$$
I_n =
\begin{cases}
I & (n=0)\\
I_{n-1}:\sqrt{I_{n-1}}=[I_{n-1}]_1 & (n>0).
\end{cases} 
$$

Every $R_{(I, n)}$ is also an Arf ring, and we have the tower 
$$R=R_{(I, 0)}\subseteq R_{(I, 1)}\subseteq \cdots \subseteq R_{(I, n)}\subseteq \cdots \subseteq \overline{R}$$   
of Arf rings. We also have $I_n \in\calX_{R_{(I, n)}}$ and 
$$I_n=\sqrt{I_n}\cdot I_{n+1}.$$
Note that $I_n$ is also an ideal in $R_{(I, n+1)}$, $I_n\subseteq I_{n+1}$,  and $I_n\cong I_{n+1}$ as $R_{(I, n+1)}$-modules. 

\end{defn}

Summing the above discussion, we have the following:

\begin{thm}\label{4.3}

$I=\left[\sqrt{I_0}\sqrt{I_1}\cdots\sqrt{I_n}\right]\cdot I_{n+1}$ for any $n\ge0$.

\end{thm}


To delete $I_{n+1}$ in the equation, we require the following works. 

\begin{lem}\label{5.1}
$I_n:I_n=I_{n+1}:I_{n+1}$ and $R_{(I, n)} \subseteq B$ for any $n\ge 0$.

\end{lem}

\begin{proof}
For any $n\ge 0$, because $I_n\cong I_{n+1}$ as $R_{(I, n+1)}$-modules, we have $I_n:I_n=I_{n+1}:I_{n+1}$. When $n=0$, it is clear that $R_{(I, 0)}=R\subseteq B$. If $n>0$, because $I_{n-1}\subseteq \sqrt{I_{n-1}}$ in $R_{(I, n-1)}$, we obtain 
$$R_{(I, n)}=\sqrt{I_{n-1}}:\sqrt{I_{n-1}}\subseteq I_{n-1}:I_{n-1}=\cdots =I_0:I_0=B$$
by Proposition \ref{3.2}.
\end{proof}

\begin{prop}\label{5.2}
We can choose $n\ge 0$ such that $R_{(I, n)}=R_{(I, \ell)}$ for every $\ell\ge n$.
\end{prop}

\begin{proof}
This follows from Lemma \ref{5.1}, because  $B$ is a Noetherian $R$-module.
\end{proof}

\begin{lem}\label{5.3}
For each $n\ge0$, the following conditions are equivalent:
\begin{enumerate}[$(1)$]
\item
$R_{(I, n)}=R_{(I, n+1)}$.
\item
$I_n$ is a principal ideal of $R_{(I, n)}$.
\item
$\sqrt{I_n}$ is a principal ideal of $R_{(I, n)}$.  
\end{enumerate}
\end{lem}
\begin{proof}
For each $n\ge 0$, we can take $a_n\in W(R_{(I, n)})$ such that $\sqrt{I_n}=\overline{a_nR_{(I, n)}}$, because $R_{(I, n)}$ is an Arf ring. Then, we have
$$R_{(I, n+1)}=\sqrt{I_n}:\sqrt{I_n}=\frac{\sqrt{I_n}}{a_n}.$$
Therefore, $R_{(I, n)}=R_{(I, n+1)}$ if and only if $\sqrt{I_n}=a_nR_{(I, n)}$, which implies that $(1)\Leftrightarrow (3)$. $(2)\Leftrightarrow(3)$ follows from Proposition \ref{3.4}.  
\end{proof}

\begin{prop}\label{5.4}
Let $n\ge 0$. If $R_{(I, n)}=R_{(I, n+1)}$, then $R_{(I, \ell)}=B$ for any $\ell \ge n$.
\end{prop}
\begin{proof}
Suppose that $R_{(I, n)}=R_{(I, n+1)}$. Then, $I_n\cong R_{(I, n)}=R_{(I, n+1)}$, because $I_n$ is a principal ideal of $R_{(I, n)}$ by Lemma \ref{5.3}.
Therefore, we have
$$R_{(I, n)}=R_{(I, n)}:R_{(I, n)}=I_n:I_n=B.$$
In contrast, because $I_n\cong I_{n+1}$ as $R_{(I, n+1)}$-modules, we also have $I_{n+1}\cong R_{(I, n+1)}$ as $R_{(I, n+1)}$-modules, which implies that $R_{(I, n+1)}=R_{(I, n+2)}$ by Lemma \ref{5.3}. Thus, we obtain $R_{(I, \ell)}=R_{(I, n)}=B$ for all $\ell \ge n$.
\end{proof}

We then have the following:

\begin{thm}\label{5.5}
We can choose $n\ge0$ such that $I_{\ell}=R_{(I, \ell)}=B$ for every $\ell\ge n$.
\end{thm}

\begin{proof}
We take $m\ge 0$ such that $R_{(I, m)}=R_{(I, m+1)}$. Then, $R_{(I, \ell)}=B$ for any $\ell\ge m$ by Proposition \ref{5.4}. Additionally, because $I_\ell=\sqrt{I_\ell}\cdot I_{\ell+1}$, we have
$$I_m\subseteq I_{m+1}\subseteq \cdots \subseteq I_\ell \subseteq \cdots \subseteq B.$$
If $I_\ell \subsetneq B$ for every $\ell\ge m$, we also have $\sqrt{I_{\ell}}\subsetneq B$, so that
$$I_\ell=\sqrt{I_\ell}\cdot I_{\ell+1}\subsetneq I_{\ell +1}.$$
This is impossible, because $B$ is a Noetherian ring.
\end{proof}

Consequently, we get the following, which is the main result of this paper.

\begin{cor}\label{5.6}
We can choose $n\ge0$ such that $I=\sqrt{I_0}\sqrt{I_1}\cdots\sqrt{I_n}$. 
\end{cor}
\begin{proof}
We choose $n\ge0$ such that $I_{\ell}=R_{(I, \ell)}=B$ for every $\ell\ge n$. Then, because $\sqrt{I_n}$ is an ideal of $B$ and $I_{n+1}=B$, we have
$$I=\left[\sqrt{I_0}\sqrt{I_1}\cdots\sqrt{I_n}\right]\cdot I_{n+1}=\left[\sqrt{I_0}\sqrt{I_1}\cdots\sqrt{I_n}\right]\cdot B=\sqrt{I_0}\sqrt{I_1}\cdots\sqrt{I_n}.$$
\end{proof}

\vspace{0.5em}

\begin{rem}\label{5.7}
Suppose $I\neq R$. We choose $n\ge0$ such that $I_{\ell}=R_{(I, \ell)}=B$ for every $\ell\ge n$ and set
$\calS=\{ i\ge0 \mid I_i\subsetneq R_{(I, i)}  \}$.
Then, $i<n$ holds for any $i\in \calS$, and we have 
$$\calS=\{0, 1, \ldots, q\}$$
where $q=\max \calS$. Hence, we have $\sqrt{I_{i}}\subsetneq R_{(I, i)}$ and 
$$\sqrt{I_{i}}=\underset{\fkp\in V(I_i)}{\prod}\fkp$$
 for any $0\le i\le q$. Therefore, corresponding to the fact that every non-zero ideal in a DVR is the power of the maximal ideal, every ideal in $\calX_R$ can be described as the products of maximal ideals in  rings $R$, $R_{(I, 1)}, R_{(I, 2)}, \cdots, R_{(I, q)}$.

\end{rem}


\section{Case where $R$ and $\overline{R}$ are local rings}\label{sec6}

In this section, we study in detail the local case of Corollary \ref{5.6}. We assume that $(R, \m)$ is an Arf local ring, and the integral closure $\overline{R}$ is also a local ring. 
For each $i\ge0$, we set 
$$A_i =
\begin{cases}
R      & \ \ (i=0)\\
A_{i-1}^{{\frak m}_{i-1}}   & \ \ (i>0)
\end{cases}
$$
where ${\frak m}_{i}= \operatorname{J}(A_{i})$ and ${\frak m}_0={\frak m}$.
Then, every $A_i$ is also a local ring with the maximal ideal $\m_i$.

Let $I\in \calX_R$ and assume that $I\neq R$. For simplicity, we write $R_{(I, i)}=R_i$. Then, because $\sqrt{I}=\m$, we have
$$R_1=\m:\m=A_1.$$
We choose $q\ge0$ such that $\calS=\{ i\ge0 \mid I_i \subsetneq R_i \}=\{ 0, 1, \ldots, q \}$. Then, we also have the following:

\begin{lem}\label{6.1}
$R_i=A_i$ for any $i\in\calS$.

\end{lem}

\begin{proof}
Because $\overline{R}$ is a local ring, $R_i$ is also a local ring for every $i\ge 0$. Therefore, for each $1\le i \le q$, we have
$$V(I_i)=\Max R_i=\{\m_i \},$$
which yields $R_i=\m_{i-1}:\m_{i-1}=A_i$. 
\end{proof}

Therefore, $\sqrt{I_i}=\m_i\in\Max R_i$ for each $0\le i\le q$ and $I_{\ell}=R_{\ell}$ for any $\ell \ge q+1$. By combining this with Corollary \ref{5.6}, we immediately obtain the following:   

\begin{thm}\label{6.2}
$I=\m_0\m_1\cdots\m_q$.
\end{thm}

\begin{rem}\label{6.3}
The decomposition of Theorem \ref{6.2} can be considered to be a generalization of the properties of DVRs for integrally closed ideals in Arf rings. Indeed, suppose that $R$ is integrally closed. Then, with our assumption, $R$ is a DVR (of course, an Arf ring). It is well known that every nonzero ideal $I$ can be written as $I=\m^{q+1}$, where $q=\max \calS$. This is the case where $\m=\m_0=\m_1=\cdots=\m_q$ in Theorem \ref{6.2}. 

\end{rem}

We assume that $\overline{R}$ is a  finitely generated $R$-module. Then, $\overline{R}$ is a DVR. Because $R$ is an Arf ring, we have 
$$\underset{i\ge0}{\bigcup} A_i =\overline{R}$$
using Theorem \ref{2.7}. Therefore, there exists an integer $N\ge 0$ such that 
$$R=A_0\subseteq A_1\subseteq \cdots \subseteq A_N=\overline{R}.$$ 
Let $I\in\calX_R$ and assume that $I\neq R$. We set $q=\max \calS$. If $q\le N$, then $I=\m_0\m_1\cdots\m_q$, so that there are only finite possibilities for $I$. 

Suppose that $q>N$. Then, $R_i=A_i=A_N=\overline{R}$ for any $N\le i\le q$, so that 
$$B=I:I=R_q=\overline{R}.$$
Therefore, because $I$ is also an ideal of $\overline{R}$, we have 
$$I=\m_0\cdots\m_{N-1}\m_{N}^{q+1-N}= a\overline{R} $$
for some $a\in I$. 

Consequently, we can completely describe every $I\in\calX_R$ when $\overline{R}$ is a finitely generated $R$-module. We obtain the following corollary: 

\begin{cor}\label{6.4}
Suppose that $\overline{R}$ is a finitely generated $R$-module. Then, the set
$$\{I\in \calX_R \mid I~~\text{is not an ideal of}~~\overline{R}  \}$$
is finite. 
\end{cor}

Closing this paper, let us note an example:

\begin{ex}\label{6.5}
Let $V=k[[t]]$ be the formal power-series ring over a field $k$, and let $R=k[[t^3, t^{11}, t^{13}]]$. Then, $R$ is an Arf ring, and
$$R=A_0\subsetneq A_1=k[[t^3, t^8, t^{10}]]\subsetneq A_2=k[[t^3, t^5, t^7]]\subsetneq A_3=k[[t^2, t^3]] \subsetneq A_4=V.$$
Let $\m_i$ be the maximal ideal of $A_i$, and let $\m=\m_0$. We set $I=\overline{t^6R}$. Then, $I=(t^6, t^{11}, t^{13})$, and we have
$$\sqrt{I}=\m=(t^3, t^{11}, t^{13}),$$
$$I_1=I:\m=\frac{I}{t^3}=_{R}\langle t^3, t^8, t^{10}\rangle= _{R_1}\langle t^3, t^8, t^{10}\rangle=\m_1=\sqrt{I_1},$$
$$R_1=\sqrt{I}:\sqrt{I}=\m:\m=A_1,$$
$$I_2=\m_1:\m_1=\sqrt{I_1}:\sqrt{I_1}=R_2=A_2,$$
where $_{R}\langle t^3, t^8, t^{10}\rangle$ denotes an $R$-module generated by $t^3, t^8, t^{10}$.
This means $q=1$. Because $I_2=R_2$, we also have 
$$R_2=R_3=\cdots =B~(=I:I),$$
and 
 $$I=(t^6, t^{11}, t^{13})=_{R}\langle t^3, t^{11}, t^{13}\rangle\cdot _{R}\langle t^3, t^8, t^{10}\rangle=_{R}\langle t^3, t^{11}, t^{13}\rangle\cdot _{R_1}\langle t^3, t^8, t^{10}\rangle=\m_0\m_1.$$
\end{ex}

\vspace{0.5em}

\begin{ac}
The author is grateful to Professor S. Goto for his helpful advice and useful comments.
\end{ac}




\begin{thebibliography}{20}



\bibitem{Arf}
{\sc C. Arf}, Une interpr\'{e}tation alg\'{e}brique de la suite des ordres de multiplicit\'{e} d'une branche alg\'{e}brique, {\em Proc. London Math. Soc.}, Series 2, {\bf 50} (1949) 256--287.

\bibitem{AS}
{\sc F. Arslan and N. Sahin}, A fast algorithm for constructing Arf closure and a conjecture.  {\em J. Algebra} {\bf 417} (2014), 148--160. 



























\bibitem{C}
{\sc E. Celikbas, O. Celikbas, C. Ciuperc, N. Endo, S. Goto, R. Isobe, and N. Matsuoka}, {\em Weakly Arf rings}, arXiv: 2006.01330, (2020).

\bibitem{CCGT}
{\sc E. Celikbas, O. Celikbas, S. Goto, and N. Taniguchi}, Generalized Gorenstein Arf rings, {\em Ark. Mat.}, {\bf 57} (2019), 35--53. 


\bibitem{G}
{\sc S. Goto}, {\em Integral closedness of complete-intersection ideals}, {\em J. Algebra}, {\bf 108} (1987), 151--160.

\bibitem{GT}
{\sc S. Greco and C. Traverso}, {\em On seminormal schemes}, {\em Compositio Mathematica}, Tome {\bf 40} (1980), no. 3, 325--365.


\bibitem{L}
{\sc J. Lipman}, Stable ideals and Arf rings, {\em Amer. J. Math.}, {\bf 93} (1971), 649--685. 












\bibitem{SH}
{\sc I. Swanson and C. Huneke}, Integral Closure of Ideals, Rings, and Modules, {\em Cambridge University Press}, 2006.


\bibitem{RGGB}
{\sc J. C. Rosales, P. A. Garc\'{i}a-S\'{a}nchez, J. I. Garc\'{i}a-Garc\'{i}a, M. B. Branco},  Arf numerical semigroups. {\em J. Algebra}, {\bf 276} (2004), no. 1, 3--12.




\bibitem{Z}
{\sc O. Zariski}, Studies in equisingularity. III. Saturation of local rings and equisingularity,  {\em Amer. J. Math.}, {\bf 90} (1968), 961--1023. 



\end{thebibliography}
\end{document}